\documentclass{amsproc}
\usepackage{amsmath,amsthm,amsfonts,amssymb,eucal}
\usepackage{hyperref}
\usepackage{url} 

\theoremstyle{plain}

\newtheorem{conjecture}{Conjecture}
\newtheorem{corollary}{Corollary}

\newtheorem{definition}{Definition}

\newtheorem{lemma}{Lemma}

\newtheorem{proposition}{Proposition}
\newtheorem{remark}{Remark}

\newtheorem{theorem}{Theorem}
\numberwithin{equation}{section}

\newtheorem*{thm}{Theorem}

\date{\today}

\def\pcf{{p.c.f.}}
\newcommand{\F}[1]{\ensuremath{\mathfrak{#1}}}
\def\cc{c_{_{\hskip-.1em\Delta}}}

\begin{document}
\title[Julia sets and gaps in the spectrum]{Disconnected Julia sets and gaps in the spectrum of Laplacians on symmetric finitely ramified fractals}

\author[K.~Hare]{Kathryn E.~Hare}
\address[K.~Hare and D.~Zhou]{Department of Pure Mathematics, University of Waterloo, Waterloo ON N2L 3G1, Canada}
\email[K.~Hare]{kehare@uwaterloo.edu}
\urladdr{\url{http://www.math.uwaterloo.ca/PM_Dept/Homepages/Hare/hare.shtml}}
\thanks{K.~Hare and D.~Zhou are partially supported by the NSERC}

\author[B.~Steinhurst]{Benjamin A.~Steinhurst}
\address[B.~Steinhurst]{Department of Mathematics, Cornell University, Ithaca, NY 14853-4201, USA}
\email[B.~Steinhurst]{steinhurst@math.cornell.edu}
\urladdr{\url{http://www.math.cornell.edu/~steinhurst/}}
\thanks{B.~Steinhurst and A.~Teplyaev are partially supported by the NSF}

\author[A.~Teplyaev]{Alexander Teplyaev}
\address[A.~Teplyaev]{Department of Mathematics, University of Connecticut, Storrs CT 06269-3009, USA}
\email[A.~Teplyaev]{teplyaev@math.uconn.edu}
\urladdr{\url{http://www.math.uconn.edu/~teplyaev/}}

\author[D.~Zhou]{Denglin Zhou}
\email[D.~Zhou]{dzhou@math.uwaterloo.ca}

\subjclass[2000]{Primary 28A80; Secondary 35P05, 35J05}
\keywords{Laplacian, fractal, spectrum, gaps, Julia set}
\thanks{This paper is in final form and no version of it will be submitted for publication elsewhere.\\ \today}

\begin{abstract}
It is known that Laplacian operators on many fractals have gaps in their
spectra. This fact precludes the possibility that a Weyl-type ratio can have
a limit and is also a key ingredient in proving that the Fourier series on
such fractals can have better convergence results than in the classical setting. 
In this paper we prove that the existence of gaps is equivalent to the total
disconnectedness of the Julia set of the spectral decimation function for the class of fully symmetric \pcf\ fractals, and for self-similar fully symmetric finitely ramified fractals with regular harmonic structure. We also formulate conjectures related to geometry of finitely ramified fractals with spectral gaps, to complex spectral dimensions, and to convergence of Fourier series on such fractals.  






\end{abstract}
\date{\today}
\maketitle

\section{Introduction}

In recent years, there have been extensive studies of Laplacian operators on self-similar fractals, both as  normalized limits of discrete Laplacians on finite graphs and as generators of a diffusion process. One prominent feature of these Laplacian operators is that there can be gaps in their spectrum. Examples include the standard Laplacian on the Sierpinski gasket \cite{FS,GRS,Te}, the $n$-branch tree-like fractals and the Vicsek sets \cite{FordSteinhurst,ZhGaps,ZhVS}, and a number of other cases \cite{eigenpapers,DS,Kr2002,KrTe03}.  The existence of gaps is an interesting phenomenon in itself as this does  not happen in the classical cases. It is also a significant issue to analysis  on fractals. For instance, the existence of gaps precludes the possibility that a Weyl-type limit can exist as the ratio must drop by a constant factor when passing through a gap.Gaps in the spectrum occur if the $\limsup \lambda_n/\lambda_{n+1} > 1$ where $\lambda_n$ are eigenvalues, see Definition \ref{def:gaps}. (Note however that, alternatively, oscillations in the spectrum  can occur when there are large multiplicities of eigenavalues but no gaps, see \cite[and references therein]{ADT,eigenpapers,BK}.) Furthermore, the existence of gaps, together with a suitable heat kernel estimate allows one to show that Fourier series on these fractals can have better convergence than the classical case. This was first observed by Strichartz in \cite{St} for the Sierpinski gasket and, as Strichartz points out, \textquotedblleft ... is the first kind of example which improves on the corresponding results in smooth analysis.\textquotedblright

We remark that there is an extensive  general theory of heat kernel estimates on fractals, which imply that the fractals we consider have Laplacians (and corresponding diffusion processes) whose heat kernel (transition probability density) satisfy the so called sub-Gaussian heat kernel estimates. The post-critically finite (\pcf) case is covered in \cite{HK}, and the latest developments and further references can be found in \cite{BBK,BBKT,GT,Ki09}. The heat kernel estimates are not used in our work, but are assumed in \cite{St}. 

Our first theorem defines a class of fractals for which the spectral decimation method describes the spectrum of the Laplacian in terms of the Julia set of a rational function. This result has appeared in the literature before but not with a self-contained proof for non-p.c.f. fractals.

\begin{theorem}\label{thm:eigs}
If the harmonic structure on a fully symmetric, finitely ramified self-similar fractal with equal resistance weights  is regular, that is $r<1$, then every eigenvalue of the Laplacian has a form 
\begin{equation}
	\lambda_k = \lim_{n \rightarrow \infty} \cc^{n}\lambda_k^{(n)} = \lim_{n \rightarrow \infty} \cc^{n_0+m}\phi_0^{m} \lambda_k^{(n_0)}
\end{equation}  
where $k \in \mathbb{N}$, 
$\lambda_k^{(n)}$ is the $k$-th eigenvalue of the Laplacian $\Delta_n$, $n = n_0+m$. The second equality holds for any $n_0 = n_0(k)$ large enough, depending on $k$ and $\lambda_k$. 
Moreover, the multiplicities can be computed according to the formulas in \cite[Proposition~4.1]{eigenpapers}.  
\end{theorem}

In \cite{ZhGaps}, one of the authors gave general criteria for the existence of gaps in the spectrum of the Laplacian on fractals that admit spectral  decimation and used this to establish the existence of gaps in many  examples. All of the examples exhibit a common feature that the Julia set of the spectral decimation function (a rational function associated with the Laplacian) is totally disconnected. This has raised  the question of whether the existence of gaps in the spectrum of the Laplacian is equivalent to the total disconnectedness of the Julia set of the spectral decimation function. The main theorem of this paper gives a large supply of functions where this criterion holds.

\begin{theorem}\label{thm2}
Under the conditions of the theorem above, there exist gaps in $\sigma(\Delta)$ if and only if  $\mathcal{J}_R$ is totally disconnected.
\end{theorem}

This result can be generalized for a larger class of finitely ramified symmetric fractals with weights, and with more elaborate combinatorial structure, but doing so would require dealing with  many technical details and  auxiliary results which are not  available in the existing literature. The main ideas and techniques behind these results come from\cite{eigenpapers,MT,Sh,ZhGaps}. One of the newest examples of fractals we consider can be found in \cite{GibsonMacKenzie}. We also give results about the location of gaps which  generalize the results in \cite{ZhGaps}.  

\subsection*{Acknowledgments} The authors thank Robert Strichartz for mentioning this problem and many very helpful discussions.

\section{Preliminary background}

\subsection{Finitely ramified and \pcf\ fractals with full symmetry}

The following is a definition of a class of fractals where spectral gaps appear naturally. 

\begin{definition}\label{def1} We say that $K$ is a fully symmetric finitely ramified self-similar set if   $K$ is a compact connected metric space with injective contraction maps $\{\psi_j\}_{j=1}^N$  such that $$K = \bigcup_{j=1}^{N} \psi_j(K)$$ and the following three conditions hold: 
 \begin{enumerate}
	 \item there exists a finite subset $V_0$  of $K$ such that $$\psi_j(K) \cap \psi_k(K) = \psi_j(V_0) \cap \psi_k(V_0)$$ for $j \neq k$ (this  intersection may be empty); 
	\item if $v_0\in V_0\cap\psi_j(K)$ then $v_0$ is the fixed point of $\psi_j$;
	\item  there is a group $\F G$ of isometries of $K$ that has a doubly transitive action on $V_0$ and is compatible with the self-similar structure $\{\psi_j\}_{j=1}^N$, which means (\cite[Proposition 4.9]{NT} and also \cite{eigenpapers,MT}) that for any $j$ and any $g\in\F G$ there exists $k$ such that $$ g^{-1} \circ \psi_j \circ g =\psi_k.$$ 
 \end{enumerate}
\end{definition}

One can see  that   $V_0$ contains at least two points if $K$ is not a singleton, which we always assume.

Post critically finite (\pcf) self-similar sets are defined in \cite{Ki,Ki2}. We will not repeat the  definition, which in general does not assume any symmetries.  For examples of \pcf\ fully symmetric  self-similar sets see  \cite{DS,FordSteinhurst,FS,Kr2002,KrTe03,MT,NT,Sh,ZhGaps,ZhVS}, while some examples of fully symmetric finitely ramified fractals that are not \pcf\ can be found in \cite{ADT,eigenpapers,T08cjm}. In the fully symmetric case one can easily obtain the following proposition.

\begin{proposition}\label{prop1} 
A  fully symmetric finitely ramified self-similar set $K$ is a  \pcf\ self-similar set if and only if for any $v_0\in V_0$ there is a unique $j$ such that $v_0\in\psi_j(K)$. 
\end{proposition} 


\subsection{Spectral decimation for graph Laplacians}\label{sub-SD}
This subsection follows \cite{eigenpapers,MT,Te}, and more information about self-similar graphs can be found in \cite{Kr2002,KrTe03}. 

Following Definition~\ref{def1}, we define recursively $$V_n = \bigcup_{j=1}^{N} \psi_j(V_{n-1})$$ and call these sets the vertices of level or depth $n$. Note that $V_n\subset V_{n+1}$ and the sets $V_n$ approximate $K$ in the sense that $K = \overline{\cup_{n=0}^{\infty} V_n}$.

There is an associated recursively defined sequence of self-similar graphs $G_n$ that have $V_n$ as their vertex set. We define $G_0$ to be the complete graph on $V_0$. Then two elements, $x,y \in V_n$ are connected by an edge in $G_n$ if $\psi_j^{-1}(x)$ and $\psi_j^{-1}(y)$ are connected in $G_{n-1}$. Note that  $G_n\not\subset G_{n+1}$ and, in fact,  one can deduce  that  $E(G_n)\cap E(G_{n+1}) = \varnothing$, where $E(G_n)$ denotes the set of edges of the graph $G_n$. In particular, in $G_1$ no two elements of $V_0$ are neighbors. 

\begin{definition}\label{def2}
Let the operator $\Delta_n$, called the discrete probabilistic Laplacian on $V_n$, be defined by
\begin{equation*}
	\Delta_nf(x) = f(x) - \frac{1}{\deg_n(x)} \sum_{(x,y) \in E(G_n)} f(y) 
\end{equation*}
where $\deg_n(x)$ is the degree of $x$ in the graph $G_n$, which may depend on $n$. 
\end{definition}

In this paper we define all Laplacians to be non-negative operators which is different by a minus sign from the usual probabilistic convention that generators of random walks and diffusion processes are non-positive. 

The matrix of $\Delta_n$ with respect to the standard basis for functions on $V_n$, ordered so that the basis elements representing $V_{n-1}$ are listed first, will be denoted $M_n$. We also denote by $I_{1,0}$ the identity matrix of size $|V_0|$, and by $I_{n+1,n}$ the identity matrix of size $|V_{n+1}\backslash V_n|$.

The matrix $M_1$ can be decomposed in the following block form
\begin{equation*}
	M_1 = \left( \begin{array}{cc} A & B \\ C & D \end{array} \right) 
\end{equation*}
where $A=I_0$ is the identity matrix of size $|V_0|$ that corresponds to the vertices $V_0 \subset V_1$. The Schur complement of $M_1$ is $S = A-BD^{-1}C$.  The spectral decimation function $R(z)$ will be calculated from the Schur complement of the matrix $M_1-zI_1$ which is the matrix valued function 
\begin{equation}\label{e-S}
	S(z) = (1-z)A - B(D-zI_{1,0})^{-1}C.
\end{equation} 

\begin{proposition}[\cite{eigenpapers,MT}]\label{prop2}
For a fully symmetric self-similar structure on a finitely ramified fractal $K$ there are unique rational scalar-valued functions $\phi(z)$ and $R(z)$ that satisfy 
\begin{equation}\label{eq:S(z)}
	S(z) = \phi(z)\left(M_0 - R(z) \right) 
\end{equation}
 and  are given by 
\begin{eqnarray*}
	\phi(z) &=& -(|V_0|-1)S_{1,2}(z)\\
	R(z) &=& 1-\frac{S_{1,1}}{\phi(z)}.
\end{eqnarray*}
\end{proposition}
Where $S_{i,j}(z)$ is the $i,j$ element of the matrix $S(z)$.
\begin{corollary}\label{cor-S}
 $\phi(z)=O\left(\frac1{|z|}\right)_{z\to\infty}$ \,and\,\,\, 
 $R(z) \geqslant O\left({|z|^2}\right)_{z\to\infty}$.
\end{corollary}
\begin{proof}
 By \eqref{e-S}, the diagonal terms of $S(z)$ grows linearly at infinity, and the off-diagonal terms tend to zero. 
\end{proof}

The initial step in spectral decimation is to relate the eigenvalues of $M_1$ back to those of $M_0$ with the help of the rational function $R(z)$, and after that continue iteratively by induction in $n$. If $v$ is an eigenvector of $M_1$ with eigenvalue $z$, then we can write $v = ( v_0, v_1')^{T}$ and
\begin{equation*}
	M_1v = \left( \begin{array}{cc} A&B\\C&D \end{array}\right) \left( \begin{array}{c} v_0\\v_1' \end{array} \right) = z \left( \begin{array}{c} v_0\\v_1' \end{array} \right),
\end{equation*}
which can be rewritten as two equations
\begin{eqnarray*}
	Av_0 + Bv_1' &=& zv_0\\
	Cv_0+Dv_1' &=& zv_1'
\end{eqnarray*}
This can be solved to give $v_1' = -(D-z)^{-1}Cv_0$, provided that $z \notin \sigma(D)$, which then implies that $S(z)v_0=0$. Note that $v_0$ is an eigenvector of $M_0$ with eigenvalue $z_0$ if and only if  $(M_0-z_0)v_0=0$, which we relate to the Schur complement  $S(z)$ by Proposition~\ref{prop2}, obtaining $z_0=R(z)$. This calculation is possible if $D-z$ is invertible and $\phi(z)\neq0$, which motivates the following definition. 

\begin{definition}
We denote the set $\sigma(D) \cup \{ z: \phi(z)=0\}$ by $E(M,M_0)$ and call it the exceptional set for the sequence of discrete Laplacians $\Delta_n$ on $G_n$.
\end{definition}

If we suppose that $z\notin E(M,M_0)$ and apply the above argument, then $z$ is an eigenvalue of $M_1$ with eigenvector $v$ if and only if $R(z)$ is an eigenvalue of $M_0$ with eigenvector $v_0$ and $v = (v_0,v_1')^{T}$, where $v_1'$ is given by 
\begin{equation*}
	v_1' = -(D-z)^{-1}Cv_0.
\end{equation*}
This implies the existence of a one-to-one map from the eigenspace of $M_0$ corresponding to $R(z)$ onto the eigenspace of $M_1$ corresponding to $z$ 
\begin{equation*}
	v_0 \mapsto v = T_0(z)v_0 = \left( \begin{array}{c} I_0 \\ -(D-z)^{-1}C \end{array} \right)v_0, 
\end{equation*}
which is called the eigenfunction extension map. 
This calculation gives a part of the spectrum of $M_1$ using only the spectrum $M_0$. It does leave open whether any of the elements of $E(M,M_0)$ are in the spectrum of $M_1$ (see \cite{eigenpapers} for more detail). Also note that while $z \not\in E(M,M_0)$ it can happen that $R(z)$ is. 

The previous  argument serves as the base case in the inductive nature of spectral decimation, and the next lemma gives the inductive steps, which use the function $R(z)$ as the spectral decimation function that relate eigenvalues from one level to the next higher level. The function $R(z)$ and the exceptional set $E(M,M_0)$ have the property that  
 if $ \Delta _{{m+1}}f=\lambda f$ and $\lambda \notin E(M,M_0)$, then 
 $ \Delta _{{m}}f|_{V_{m}}=R(\lambda )f|_{V_{m}}$ and vice versa. 
We denote $M_n$ the matrix representation for $\Delta_n$ given in a block form as 
\begin{equation*}
	M_n = \left( \begin{array}{cc} A_n &B_n\\C_n&D_n \end{array} \right)
\end{equation*}
corresponding to the representation of $V_n$ as $V_{n-1} \cup \left(V_n \setminus V_{n-1} \right)$, and let $P_{n-1}:V_n \rightarrow V_{n-1}$ be the restriction operator.

\begin{lemma}[\cite{eigenpapers,MT,Te}]\label{lemma:extension}
For all $n >0$ we have
\begin{equation*}
	P_{n-1} (M_n - z)^{-1}P_{n-1}^{*} = \frac{1}{\phi(z)}(M_{n-1} - R(z))^{-1}.
\end{equation*}
Suppose that $z_n \not\in E(M,M_0)$. Then $z_n$ is an eigenvalue of $M_n$ with an eigenvector $v_n$ if and only if $z_{n-1} = R(z_n)$ is an eigenvalue of $M_{n-1}$ with eigenvector $v_{n-1}$  and $v_n = (v_{n-1}, v_n')^{T}$ where
\begin{equation}
	v_n' = -(D_n-z_n)^{-1}C_n v_{n-1}.
\end{equation}
\end{lemma}

We will refer to $v_n$ as an extension of $v_{n-1}$ from $V_{n-1}$ to $V_n$. This lemma  does not explain the status of potential eigenvalues from the set of exceptional values  $E(M,M_0)$. The question of when exceptional values are eigenvalues is  resolved in \cite{eigenpapers}. 
 \begin{remark}\label{remark-ds}
 It is known, by \cite[Lemma 4.9]{MT} and \cite{Sh}, that 
 $$R(0) = 0 \text{ \ \ and \ \ } \cc =R^{\prime }(0) >1.$$ 
 Therefore, $0$ is a repulsive fixed point of $R$, which is 
 important for the complex dynamics associated with this function. 
 
 Moreover,  $\cc$ is also called the \emph{Laplacian normalization constant}, which appears in Subsection~\ref{sub-SD-K}, and there are relations $$\cc=Nc\text{\ \ and\ \ }d_R=\dfrac{\log N}{\log c},$$
 where $c=\frac1r$ is the conductance scaling factor, $r$ is the resistance scaling factor, 
 and $d_R$ is the Hausdorff dimension in the effective resistance metric. 
 
 Often in the literature there is the relation  $$d_S=\dfrac{2\log N}{\log Nc}=\dfrac{2d_R}{d_R+1},$$ where $d_S$ is the so-called spectral dimension (for more detail  see Subsection~\ref{sub-res} below and \cite{KL1,St03jfa}). 
 In particular, \cite{St03jfa} mentions that the notion of spectral dimension is a
misnomer because the coefficient $2$ implicitly assumes that the Laplacian is an operator of order two, which may not be justified in fractal case. Furthermore, Strichartz argues that a natural notion of the order of the Laplacian is $d_R+1$. 
\end{remark}

\subsection{Spectral decimation for the Laplacian on $K$}\label{sub-SD-K}

The standard Laplacian operator $ \Delta$ is defined by 
\begin{equation}\label{e-Delta}
	\Delta u=\lim_{n\rightarrow \infty }\cc^{n}\Delta _nu(x) 
\end{equation}
if the limit exists. Here 
 $x\in V_{\ast }$,   $\cc= R'(0)$ is the same as above, and the sequence of difference operators $\{\Delta _n\}_{n=0}^{\infty }$ acting on functions defined on $V_{m}$ is defined in  Definition~\ref{def2}. 
 
 \begin{definition}\label{def-Lp}
 The   \emph{continuous Neumann Laplacian} is  defined for all functions $u$ for which the limit \eqref{e-Delta} exists for all  $x\in V_{\ast }$, and there is a continuous function $f$ such that $\Delta u(x)=f(x)$ for $x\in V_{\ast }$.  

 The   \emph{continuous Dirichlet Laplacian} is  defined for all functions $u$, vanishing on $V_0$,  for which the limit \eqref{e-Delta} exists for all  $x\in V_{\ast }\backslash V_0 $, and there is a continuous function $f$, vanishing on $V_0$, such that $\Delta u(x)=f(x)$ for $x\in V_{\ast }\backslash V_0$.  
 \end{definition}
 
 Standard references for the Laplacian on \pcf\ fractals are \cite{Ki,Ki2},   \cite{StBook} for an introduction,  and \cite{Ki3} for a very general context. It will be explained below in Subsection~\ref{sub-res} why the Neumann and Dirichlet Laplacians are self-adjoint.

Let $\phi _{0},...,\phi _{L}$ denote the partial inverses of $R$, where $\phi _{0}$ is the partial inverse of $R$ with $0$ in its range, and $\phi_{0}^{(n)}$ be the $n$-th composition power of $\phi _{0}$. Note that  $L$, the degree of the rational function $R(z)$, does not have to be equal to $N$, which is the number  of contraction maps in 
Definition~\ref{def1}.  Given $w=w_{n}...w_{1}$, a word of length $n=\left\vert w\right\vert $ on the letters $0,...,L$, we put $\phi _{w}=\phi _{w_{n}}\circ \cdots \circ \phi _{w_{1}}$. 
Since $R'(0) > 1$,  
 $\phi _{0}(x)<x$ for $x < \epsilon$, and so $\lim_{n\rightarrow \infty }\cc^{n}\phi _{w}(x)$ with $w=w_{n}...w_{1}$ exists if and only if there is a word $v$ and an integer $n_0$ such that for all $n\geq n_o$, $\phi _{w}=\phi _{0}^{(n)}\circ \phi _{v}$ (see \cite{Sh,ZhGaps,ZhVS} formore detail). 
 
 \begin{definition}\label{def-sd}
A Laplacian $\Delta$  is said to admit \emph{spectral decimation} if all its eigenvalues   are of the form  
\begin{equation}\label{eq:rawlimit}
\lambda 
=	\cc^{i}\lim_{n\rightarrow \infty }\cc^{n+j}\phi _{0}^{(n)}\circ \phi_{w}(x),
\end{equation}
where $x\in \sigma(\Delta_i) \cup E(M,M_0),\ i\in \mathbb{N\cup \{}0\}\text{ and }\left\vert w\right\vert =j$.
\end{definition}

In other words, the spectrum of the Laplacian, $\sigma ( \Delta ),$ is approximated by the spectrum of $ \Delta _n$, $\sigma ( \Delta_n),$ scaled by the Laplacian renormalization constant $\cc^{n}$. The spectrum of  $ \Delta _n$ can be computed by spectral decimation (see Subsection~\ref{sub-SD}). The Sierpinski gasket, $n$-branch tree-like fractals, the fractal 3-tree and Viscek sets are all examples of fractals that admit spectral decimation. For further background on spectral decimation we refer the reader to \cite{eigenpapers,FS,Sh,StBook,ZhGaps}.

\subsection{Self-similar resistance forms and self-adjoint Laplacians}\label{sub-res}
This discussion of resistance forms is largely taken from \cite{Ki2,Ki3}. 
\begin{definition}
Let $X$ be a set. A pair $(\mathcal{E},\mathcal{F})$ is a resistance form on $X$ if it satisfies the following conditions.
\begin{enumerate}
	\item $\mathcal{F}$ is a linear subspace of $l(X)$ containing constants and $\mathcal{E}$ is a non-negative symmetric quadratic form on $\mathcal{F}$. Moreover $\mathcal{E}(u,u) =0$ if and only if $u$ is constant on $X$.
	\item Let $\sim$ be an equivalence relation on $\mathcal{F}$ defined by $u \sim v$ if and only if $u-v$ is constant on $X$. Then $(\mathcal{F}/\sim, \mathcal{E})$ is a Hilbert space.
	\item For any finite subset $V \subset X$ and for any $v \in l(V)$ there exists $u \in \mathcal{F}$ such that $u|_V =v$.
	\item For any $p,q \in X$,
	\begin{equation}
	R_{(\mathcal{E},\mathcal{F})}(p,q) =	\sup\left\{ \frac{|u(p) - u(q)|^{2}}{\mathcal{E}(u,u)} : u \in \mathcal{F}, \mathcal{E}(u,u)>0  \right\}
	\end{equation}
	is finite.  
	\item If $u \in \mathcal{F}$, then $\bar{u} \in \mathcal{F}$ and $\mathcal{E}(\bar{u},\bar{u}) \le \mathcal{E}(u,u)$, where $\bar{u} = (u \wedge 0) \vee 1$ is the normal contraction of $u$. 
\end{enumerate}
\end{definition}

Note that conditions $(1),(2),$ and $(5)$ are the conditions that are also used in the defintion of  a Dirichlet form, but here there is no reference $L^{2}$ space in the background. The quantity $R_{(\mathcal{E},\mathcal{F})}(p,q)$ is a metric called the effective resistance metric associated with the resistance form $(\mathcal{E},\mathcal{F})$. In this metric elements of $\mathcal{F}$ are $1/2$-H\"older continuous. Moreover, if $B\subset X$ is chosen as the set where zero (Dirichlet) boundary  conditions are imposed, then the Green's function $g_B(x,y)$ is always continuous outside the diagonal (see \cite[Theorem 4.5]{Ki3}). 

The definition above does not involve any self-similarity. If we deal with a self-similar fractal, such as in Definition~\ref{def1}, then, according to \cite{Ki,Ki2}, a resistance form on $V_*=\cup_{n\geq0}V_n$ is 
 self-similar with resistance weights $r_j$ if  
\begin{equation}
	\mathcal{E}(u,u) = \sum_{j=1}^{N} \frac{1}{r_j} \mathcal{E}(\psi_j^{*}(u),\psi_j^{*}(u))
\end{equation}
for any $u \in \mathcal{F}$. For such resistance forms the maps $\psi_j$ are asymptotic contraction maps in the effective resistance metric with contraction ratio $r_j$. 

In our paper we always assume that   all resistance weights are equal, that is $$r_i=r_j=r \text{ \ for all \ }i,j=1,...,N.$$ 
In addition, we assume that 
the resistance form  
is \emph{regular}, that is $$r<1,$$ which corresponds to the case $$d_S=\frac{2\log{N}}{\log{N}-\log{r}} < 2$$  according to \cite{Ki,Ki2,KL1,Kumagai-1993}, where $d_S$ is the so called spectral dimension (see Remark~\ref{remark-ds} and \cite{St03jfa} for a discussion of this).




\begin{thm}[Kigami \cite{Ki3}] If $r<1$ then the Dirichlet and Neumann Laplacians are self-adjoint, have discrete spectra and, moreover, the Dirichlet Laplacian has a continuous Green's function $g(x,y)$.\end{thm}

This result relies primarily on Theorem~8.13 in \cite{Ki3}, which says that $\Delta$ has compact resolvent  under the conditions that measure on $K$ is non-atomic and that the effective resistance metric is integrable over $K$. Thus $\Delta$ has a discrete spectrum. Finitely ramified fractals with regular harmonic structure have a natural non-atomic self-similar probability measure and have finite diameter in the effective resistance metric. Thus the spectrum of the Dirichlet or Neumann Laplacian $\Delta$ can be written as $0\le \lambda_0 \le\lambda_1 \le \lambda_2 \le \cdots \le \lambda_k \le \cdots$. Moreover, in the Dirichlet case we have $0=\lambda_0<\lambda_1$. 

In the case of a fully symmetric p.c.f. with equal resistance weights, one always has $r < 1$ and  Shima \cite{Sh} has shown that the spectral decimation method will produce the spectrum of $\Delta$, which is described in Subsection~\ref{sub-SD-K}. Here we extend this result for finitely ramified fractals with regular harmonic structure. We now prove Theorem \ref{thm:eigs}.

\begin{thm}
If the harmonic structure on a fully symmetric, finitely ramified self-similar fractal with equal resistance weights  is regular, that is $r<1$, then every eigenvalue of the Laplacian has a form 
\begin{equation}
	\lambda_k = \lim_{n \rightarrow \infty} \cc^{n}\lambda_k^{(n)} = \lim_{n \rightarrow \infty} \cc^{n_0+m}\phi_0^{m} \lambda_k^{(n_0)}\label{eq:eigenvalue}
\end{equation}  
where $k \in \mathbb{N}$, 
$\lambda_k^{(n)}$ is the $k$-th eigenvalue of the Laplacian $\Delta_n$, $n = n_0+m$. The second equality holds for any $n_0 = n_0(k)$ large enough, depending on $k$ and $\lambda_k$. 
Moreover, the multiplicities can be computed according to the formulas in \cite[Proposition~4.1]{eigenpapers}.  
\end{thm}

\begin{proof} 
We prove this result in four steps to 
clearly delineate where the assumption on the Green's function and 
spectral dimension is used, also because the first two  
steps have been discussed before but not collected into a single result. 

\def\lL{\Lambda}

\emph{Step 1:}   We show that $\lambda_k = \lim_{n \rightarrow \infty} \cc^{n}\lambda_k^{(n)}$. Using Lemma \ref{lemma:extension} if $f_n$, a function on $V_n$, is an eigenfunction of $\Delta_n$ with eigenvalue $\lambda^{(n)}_k$ there is an extension $f_{n+1}$ on $V_{n+1}$ that is an eigenfunction of $\Delta_{n+1}$ with eigenvalue $\lambda_k^{(n+1)}$ and by induction there is a continuous function $f$ on $K$ with the property that $f|_{V_n}$ is an eigenfunction of $\Delta_n$ with eigenvalue $\lambda_k^{(n)}$. 
Then
\begin{eqnarray}
	\Delta f &=& \lim_{n \rightarrow \infty} \cc^{n}\Delta_n f|_{V_n}\\
	&=& \lim_{n \rightarrow \infty} \cc^{n} \lambda_k^{(n)} f|_{V_n}\\
	&=& \lambda_kf.
\end{eqnarray}
Where $\lambda_k = \lim_{n\rightarrow \infty} \cc^{n}\lambda_k^{(n)}.$ 
So any sequence $\lambda_k^{(n)}$ of eigenvalues of $\Delta_n$ will produce an eigenvalue of $\Delta$ along this scaled limit, if it exists. These are called ``raw eigenvalues'' by Shima \cite{Sh} and Kigami \cite{Ki,Ki3}.
It is in this step that we use the assumption on the Green's function. If the Green's function is continuous then Theorem 3.5 in \cite{Sh} states that all the eigenvalues of $\Delta$ are ``raw eigenvalues'' in the post critically finite case. Kigami \cite{Ki3} provides enough background for the extension of the claim to self-similar sets where the Laplacian defines a resistance form, such as fully symmetric finitely ramified fractals with a harmonic structure is regular. 

\emph{Step 2:} Let $\lL \ge 0$. Using the mechanics of spectral decimation given above, and in more detail in \cite{eigenpapers}, $\sigma(\Delta_n)$ is calculated using inverse images of $\sigma(\Delta_0)$ and  $E(M,M_0)$ under $R(z)$. There may be several branches of of $R^{-1}$ denoted $\phi_i$ but for $n= n_0 + m$ for $n_0$ large enough only $\cc^{n_0}\cc^{m}\phi_0^{m}(\sigma(\Delta_{n_0}))$ will intersect $[0,\lL)$. The value of $n_0$ 
can be computed directly from $R$, and depends on $\lL$. But elements of $\sigma(\Delta_{n_0})$ are given by spectral decimation to be $\cc^{n_0}\phi_w(\Delta_0)$ where $w$ is any word of length $n_0$. This gives eigenvalues for $\Delta$ of the form of (\ref{eq:eigenvalue}) which is equivalent to (\ref{eq:rawlimit}). So the eigenvalues less than $\lL$ are of the form claimed. 


\emph{Step 3:}  Let $\epsilon > 0$. Observe that $\cc^{-n}\sigma(\cc^{n}\Delta_n) = \sigma(\Delta_n)$. Thus $\sigma(\Delta) \cap \cc^{-n}[0,\lL) \subset [0,\epsilon)$. For $n$ large enough $\epsilon$ can be taken to be less than the least exceptional value. And for any $\lL$ the starting level $n_0$ can be chosen high enough. Hence all eigenvalues are of the claimed form since $\lL$ is arbitrary. 

\emph{Step 4:}  The multiplicities are computed  in \cite[Theorem 1.1]{eigenpapers} and, in particular, this theorem  shows that, for a given $k$, the multiplicity of $\lambda_k$ is the same as that of $\lambda_k^{(n)}$ for all $n$ large enough (depending on $k$ as above).
\end{proof}

It is often convenient to assume that $n_0$ is the smallest integer with the properties described above, but the claims are true if $n_0$ is replaced with any integer between the smallest possible $n_0$ and $n$ in the proof above. 

R. Grigorchuk, V. Nekrashevych, and Z. Sunic have used spectral decimation to calculate the spectrum of a Laplacian on Julia sets \cite{GNS}. In this setting the mapping $R(z)$ is determined from the dynamics of the self-similar group as given by the iterated monodromy gorup. The theory presented here does not cover this case. The work by C. Sabot \cite{Sabot2003} uses a similar renormalization method with rational maps of several variables to produce spectral information about Laplacians that are similar to those considered here.

\subsection{The Julia set and the graph Laplacians}\label{sub-j}
According to the classical theory \cite{Br,CG,Milnor}, the Julia set of the spectral decimation function $R$, denoted $\mathcal{J}_R$, is given by 
\begin{equation*}
	\mathcal{J}_R=closure\bigcup_{n\geq 0}R^{-n}(0),
\end{equation*}
where $R^{-n}(0)$ are pre-images of $0$ of order $n$ (because $0$ is a repulsive fixed point). 
Furthermore, according to \cite{MT},   
we have that  
\begin{equation*}
	\mathcal{J}_R\subseteq \sigma (\Delta _{\infty })\subseteq \mathcal{J}_R\cup \mathcal{D}_{\infty } \cap \mathbb{R}^{+},
\end{equation*}
where $\Delta_{\infty} = \lim_{n \rightarrow \infty} \Delta_n$ is the discrete probabilistic Laplacian on an  infinite self-similar graph
and 
$\mathcal{D}_{\infty }\setminus \mathcal{J}_R$ contains only isolated points. Here 
 $$\mathcal{D}_{n}=\bigcup_{m=0}^{n}R^{-m}(E(M,M_0) \cup \sigma (\Delta _{0})),$$ and $$\mathcal{D}_{\infty }=\bigcup_{n=0}^{\infty }\mathcal{D}_{n},$$ where the unions are increasing in $n$. Moreover
$$\sigma (\Delta _{n})\subseteq \mathcal{D}_{n} \text{\ \ and\ \ }  \sigma (\Delta _{n})\subseteq \sigma (\Delta _{\infty}) .$$



\section{Gaps in the spectrum}In this section we prove that the existence of gaps can be characterized by
the Julia set.

It is known that spectra of Laplacian operators on many fractals have gaps. Investigation of the existence of gaps is important to analysis on fractals because of its many interesting applications, as mentioned in the introduction. Some criteria and examples are given in \cite{ZhGaps} and \cite{ZhVS}, although the verification can be tedious. In the next section we will derive a simple and easy to apply criterion based on the total disconnectedness of the Julia set of the spectral decimation function, generalizing the results of \cite{ZhGaps}.

\begin{definition}\label{def:gaps}
For a given infinite sequence $0\le \alpha_1\leq \alpha_2\leq \cdots\leq
\alpha_k\leq \cdots$, we say that there exist gaps in the sequence if $%
\displaystyle{\limsup_{k\geq1} \frac{\alpha_{k+1}}{\alpha_k}}>1$.
\end{definition}

By Theorem \ref{thm:eigs} the study of ratios of eigenvalues of the Laplacian on fully symmetric finitely ramified fractals can be reduced to the corresponding limit of ratios of eigenvalues of the discrete Laplacians on finite graphs $G_{n}$ that approximate the fractal. This is because any ratio of eigenvalues in $\sigma ( \Delta )$ will have the form 
\begin{eqnarray}
	\frac{\cc^{i}\lim_{n\rightarrow \infty }\cc^{n+\left\vert w\right\vert}\phi _{0}^{(n)}\phi _{w}(\beta _{j})}{\cc^{l}\lim_{n\rightarrow \infty}\cc^{n+\left\vert v\right\vert }\phi _{0}^{(n)}\phi _{v}(\beta _{k})} &=&\lim_{n\rightarrow \infty }\frac{\phi _{0}^{(n-i-\left\vert w\right\vert)}\phi _{w}(\beta _{j})}{\phi _{0}^{(n-l-\left\vert v\right\vert )}\phi_{v}(\beta _{k})} \\
	&=&\lim_{n\rightarrow \infty }\frac{\phi _{0}^{(n)}\phi _{w}(\beta _{j})}{\phi _{0}^{(n+r)}\phi _{v}(\beta _{k})}\label{eq:eigenform}
\end{eqnarray}
for some integer $r$ and $\beta _{j},\ \beta _{k}\in \sigma(\Delta_0)$. Both the numerator and denominator in the last ratio are eigenvalues of discrete Laplacians on graphs approximating the fractal. Hence it is sufficient to consider all ratios of eigenvalues in $\bigcup_{n}\sigma ( \Delta _n)$ (without powers of $\cc$) in order to show the existence of gaps in the spectrum of Laplacians on fractals.


According to Subsection~\ref{sub-j} and Corollary~\ref{cor-S}, the following lemma applies for any spectral decimation function of a finitely ramified fully symmetric fractal, always a rational function.

\begin{lemma}\label{lem: JuliaIntersection} Let $R$ be a spectral decimation function, rational and of degree at least 2 with an attractive fixed point at infinity and such that 
$R(0)=0$, $R'(0)>1$, and the Julia set $\mathcal{J}_R$ 
of $R$ is real and nonnegative. If   $I=[0,a]$ is  the convex hull of $\mathcal{J}_R$ 
then the following is true:  
\begin{enumerate}
\item  $R^{-n-1} (I)\subset R^{-n} (I)$ for all $n\ge 0$; 
	\item $\mathcal{J}_R=\bigcap_{n=0}^\infty R^{-n} (I)$; 
	\item either \,$\mathcal{J}_R=I$\, or \,$\mathcal{J}_R$\, is
totally disconnected; 
	\item $\mathcal{J}_R=I$\,  if and only if \,$R^{-1}(I)$\, is connected; 
\item either $R(a)=0$ or $a$ is the largest fixed point of $R$.
\end{enumerate}
\end{lemma}

\begin{proof} By the classical complex dynamics theory (see \cite{Br,CG,Milnor}), the assumptions of this lemma  imply that there is a single Fatou component attracted to infinity, and the action of $R$ on its Julia set is hyperbolic. Let $b$ be the smallest positive number such that $R^{-1}([0,b])\subseteq[0,b]$, which exists because infinity is an attractive fixed point. We have that $R^{-1}([0,b])$ is a finite collection of intervals, which easily imply that either $R(b)=0$ or $b$ is a repulsive or indifferent fixed point of $R$. However the Sullivan classification of Fatou components excludes the possibility of an indifferent fixed point in the Fatou component, and  hence $a=b$  and all the claims of the lemma follow.
\end{proof}

Our main result is Theorem \ref{thm2}. 

\begin{thm}
Under the conditions of Theorem~\ref{thm:eigs}, there exist gaps in $\sigma(\Delta)$ if and only if  $\mathcal{J}_R$ is totally disconnected.
\end{thm}

\begin{proof}
First, suppose  $\mathcal{J}_R$ is totally disconnected. Then, following Lemma~\ref{lem: JuliaIntersection}, the set $I_n= R^{-n} (I)$ is a finite collection of closed intervals which cover the Julia set and decrease as $n$ increases (because of the hyperbolicity). Hence there is an $n_0$ such that one of the intervals that constitute $I_{n_0}$ has the form $[0,\epsilon]$, where $\epsilon$ is smaller than any point in the exceptional set. Moreover, if $n_0$ is large enough then 
$[0,\epsilon]$ is contained in the domain of the branch $\phi _{0}$ of the partial inverses of $R$. It is easy to see that $[0,\epsilon]\cap I_{n+1}$ is a union of at least two closed nonintersecting intervals. 
We have that 
\begin{equation}\label{e-int} 
	\bigcup_{m=0}^n R^{-m} \big(\sigma(\Delta_0)\cup E(M,M_0)\big) \,\backslash\,  \mathcal{J}_R
\end{equation} 
consists of isolated points that accumulate to $\mathcal{J}_R$, unless the set in \eqref{e-int} is empty, because 
$\mathbb C\,\backslash\,  \mathcal{J}_R$ is attracted to infinity by the iterations of $R$. 

Therefore there are positive numbers $\alpha_0,\beta_0 \in (0,\epsilon]$ such that $\alpha_0<\beta_0$ and the interval $(\alpha_0,\beta_0)$ does not intersect 
 $\sigma (\Delta _{\infty })$. But then there is another interval $(\phi_0(\alpha_0),\phi_0(\beta_0)) = (\alpha_1,\beta_1)$ 
 that does not intersect  $\sigma (\Delta _{\infty })$ and so forth. 

The ratio $\frac{\alpha_k}{\beta_k}$ is for large enough $k$ bounded below uniformly in $k$,   since $\alpha_k, \beta_k \rightarrow 0$ as $k\to0$ and $\phi_0(z)$ is asymptotically linear near zero. 
(The starting interval could have been chosen close enough to zero so that the final fraction bounding the ratios is at least $\frac{1}{2}$ since $\mathcal{J}_R$ is totally disconnected and compact.) 
Then we multiply the spectrum by $\cc^n$ to obtain gaps in the spectra of $\cc^n\Delta_n$ which are uniform in $n$. Therefore the spectrum of $\Delta$ has gaps because of Theorem~\ref{thm:eigs}. 

In the opposite direction, suppose there are gaps in the spectrum but $\mathcal{J}_R=I$ is an interval (Lemma~\ref{lem: JuliaIntersection} implies that if $\mathcal{J}_R\neq I$ then it is totally disconnected). 
By the reasoning very similar to that   given above, the gaps in the spectrum imply that 
$\sigma (\Delta _{\infty })$ is not dense in $I$, which contradicts to the main result, Theorem 5.8, in \cite{MT}.  
\end{proof}

\begin{conjecture}\label{c1}
We conjecture that,  under the conditions of Theorems~\ref{thm:eigs}  and~\ref{thm2},  there are no gaps (which means $\mathcal{J}_R=I$  is an interval) if and only if $R(z)$ is a Chebyshev polynomial, up to  trivial constants. Furthermore, we conjecture that this happens if and only if the fractal $K$ is one of the so-called  Barlow-Evans fractals based on a unit interval (see \cite{BE,Steinhurst-BE,ST-BE} for more detail). This means that the self-similar structure of $K$ is based on an interval, such as in the case of  Diamond fractals of \cite[Section~7]{eigenpapers}, \cite{ADT}, and \cite{HK-2010}. Relevant information can be found in \cite[Section~9]{Kr2002}, \cite[Section~8]{KrTe03} and \cite{GW97,Grabner,DGV2,DGV}. Note that our results already show that  the complex dynamics of $R$ is conjugate to that of a Chebyshev polynomial if there are no gaps. 
\end{conjecture}

\begin{conjecture}
We conjecture that,  under the conditions of Theorems~\ref{thm:eigs}  and~\ref{thm2}, there are nontrivial complex dimensions if and only if $K$ is not homeomorphic to an  interval (see references in Conjecture~\ref{c1} and  \cite{LvF2006,ST-merozeta} for more detail). 
\end{conjecture}

\begin{conjecture}
We conjecture that the results of \cite{St} are true under the conditions of Theorems~\ref{thm:eigs}  and~\ref{thm2} even without heat kernel estimates (see \cite{HK,BBK,BBKT,GT,Ki09} for more detail). 
\end{conjecture}

\begin{remark}
The case of the pentagasket, which is formed by an iterated function system that transforms a pentagon to five scaled pentagons which intersect each other at single points, suggests that there is a larger class of fractals which have gaps in their spectrum. There is no current version of spectral decimation which applies on the pentagasket and allows for explicit computations, however available numerical evidence suggests that the Laplacian on the pentagasket has gaps in it's spectrum \cite{ASST}.
\end{remark}

\section{An improved criterion for gaps in the spectrum} 

We conclude with a generalization of the criteria for gaps found in \cite{ZhGaps}, which also demonstrates where gaps are located. In this section  the members of the exceptional set are also called forbidden eigenvalues.

\begin{theorem}\label{thm:crit}
Suppose $b$ is a real number dominating all the forbidden eigenvalues and
suppose $R^{-1}[0, b]\subseteq [0, b]$. Let $\{\phi
_{j}\}_{j=0}^{L} $ be the partial inverses of $R$ ordered so that $\max \phi
_{i}\leq \min \phi _{i+1}$. Assume $\phi _{0}(b)<b$ and $\phi _{0}$ is
strictly convex. There exist gaps in the spectrum of the Laplacian if for
some $0\leq J<L$, 
\begin{equation*}
M\equiv \max \phi _{J}\mid _{[0, b]}<\min \phi _{J+1}\mid _{[0, b]}\equiv m.
\end{equation*}%
In fact, there are a bounded number of elements of the spectrum in the
intervals $(A_{k}, B_{k})$ where 
\begin{eqnarray*}
A_{k} &=&\cc^{k}\lim_{m\rightarrow \infty }\cc^{m}\phi
_{0}^{(m-1)}(M)=\lim_{m\rightarrow \infty }\cc^{m}\phi _{0}^{(m-k-1)}(M) \\
B_{k} &=&\cc^{k}\lim_{m\rightarrow \infty }\cc^{m}\phi
_{0}^{(m-1)}(m)=\lim_{m\rightarrow \infty }\cc^{m}\phi _{0}^{(m-k-1)}(m).
\end{eqnarray*}

\begin{proof}
Lemma 12 of \cite{ZhGaps} shows that the strict convexity of $\phi _{0}$ and the assumption that $\phi _{0}(b)<b$ imply that $A_{k}$ and $B_{k}$ are well defined and $A_{k}<B_{k}$. Of course, $B_{k}/A_{k}=B_{0}/A_{0}>1$.

Thus each eigenvalue is of the form $x=\cc^{i}\lim_{m\rightarrow \infty}\cc^{m}\phi _{0}^{(m-j)}\phi _{v}(z)$, where $i,\ j\in \mathbb{N}\cup\{0\},\ z\in E(M,M_0)$, $|v|=j\text{ and if }j\neq 0$, then $v=v_{j}\dots v_{1}$ where $v_{j}\neq 0$. Equivalently $x=\lim_{m\rightarrow\infty }\cc^{m}\phi _{0}^{(m-i-j)}\phi _{v}(z)$.

Case 1: $j\neq 0$, say $i+j=k+1$ for $k\in \mathbb{N}\cup \{0\}.$ If $v_{j}\geq J+1$, then $\phi _{v_{j}}\circ \cdots \circ \phi_{v_{1}}(z)=\phi _{v_{j}}(z^{\prime })\geq \min \phi_{J+1}.$ Hence 
\begin{equation*}
	x=\lim_{m\rightarrow \infty }\cc^{m}\phi _{0}^{(m-k-1)}\phi_{v_{j}}(z^{\prime })\geq \lim_{m\rightarrow \infty }\cc^{m}\phi_{0}^{(m-k-1)}(m)=B_{k}.
\end{equation*}%
We can also write 
\begin{equation*}
	x=\lim_{m\rightarrow \infty }\cc^{m}\phi _{0}^{(m-k-2)}\phi _{0}(\phi_{v_{j}}(z^{\prime }))\leq \lim_{m\rightarrow \infty }\cc^{m}\phi_{0}^{(m-k-2)}\max \phi _{0},
\end{equation*}%
and this is clearly bounded by 
\begin{equation*}
	\lim_{m\rightarrow \infty }\cc^{m}\phi _{0}^{(m-k-2)}(M)=A_{k+1}.
\end{equation*}%
Thus $x\in \lbrack B_{k}, A_{k+1}]$.

Otherwise, $v_{j}\leq J$. Then $\phi _{v_{j}}\circ \cdots \circ \phi_{v_{1}}(z)=\phi _{v_{j}}(z^{\prime })\leq \max \phi_{v_{j}}\leq \max \phi_{J}$, so 
\begin{equation*}
	x\leq \lim_{m\rightarrow \infty }\cc^{m}\phi _{0}^{(m-k-1)}(M)=A_{k}.
\end{equation*}%
Furthermore, because $v_{j}\geq 1,$ 
\begin{eqnarray*}
	B_{k-1} &=&\lim_{m\rightarrow \infty }\cc^{m}\phi _{0}^{(m-k-1)}\phi_{0}(m)\leq \lim_{m\rightarrow \infty }\cc^{m}\phi _{0}^{(m-k-1)}(\max \phi _{0}) \\
	&\leq &\lim_{m\rightarrow \infty }\cc^{m}\phi _{0}^{(m-k-1)}(\min \phi_{1})\leq x.
\end{eqnarray*}%
Hence $x\in [B_{k-1}, A_{k}]$. Consequently, if $x$ is of the first type then $x\not\in \cup (A_{k}, B_{k})$.

Case 2: $j=0.$ We first note that if $z=0,$ then $x=0$ and $x\not\in (A_{k}, B_{k})$, so we assume otherwise. The strict convexity of $\phi _{0}$ ensures 
\begin{equation*}
	\frac{\phi _{0}(x)}{x}\leq \frac{\phi _{0}(b)}{b}\leq \lambda <1,
\end{equation*}%
for some $\lambda $. Thus $\phi _{0}^{(n)}(z)\leq \lambda ^{n}z\rightarrow 0$ for all $z\in [0, b]$ and therefore we can choose $n$ such that $\phi _{0}^{(n)}(m)<z^{\ast }$, where $z^{\ast } :=\min E(M,M_0)$. We claim that $x=\lim_{m\rightarrow \infty }\cc^{m}\phi_{0}^{(m-i)}(z)\not\in (A_{k}, B_{k})$, if $i\leq k$ or $i>k+n$.

To prove the claim, we first suppose $i=k-s$ for some $s\geq 0$. Then 
\begin{equation*}
	x=\lim_{m\rightarrow \infty }\cc^{m}\phi _{0}^{(m-k-1)}\phi_{0}^{(s+1)}(z)\leq \lim_{m\rightarrow \infty }\cc^{m}\phi_{0}^{(m-k-1)}\phi _{0}(z)\leq A_{k}.
\end{equation*}%
If $i=k+n+s$ for some $s\geq 1$, then by the definition of $z^{\ast }$, 
\begin{equation*}
	x\geq \lim_{m\rightarrow \infty }\cc^{m}\phi _{0}^{(m-k-n-s)}(z^{\ast }).
\end{equation*}%
By our choice of $n$, the last term above dominates  
\begin{equation*}
	\lim_{m\rightarrow \infty }\cc^{m}\phi _{0}^{(m-k-n-s)}\phi_{0}^{(n)}(m)\geq \lim_{m\rightarrow \infty }\cc^{m}\phi_{0}^{(m-k-n-s)}\phi _{0}^{(n+s-1)}(m)=B_{k}.
\end{equation*}%
This proves our claim and therefore the only forbidden eigenvalues that can
lie in the interval $(A_{k}, B_{k})$ are those of the form $\lim_{m\rightarrow \infty }\cc^{m}\phi _{0}^{(m-i)}(z)$, where $k<i\leq k+n $, and $z\in E(M,M_0)$.

It follows that the theorem holds with the bound on the number of eigenvalues in any interval $(A_{k},B_{k})$ being at most $n|E(M,M_0)|$, where $|E(M,M_0)|$ is the cardinality of $E(M,M_0)$. Hence there exists a subinterval $(c_{k},\ d_{k})\subseteq (A_{k}, B_{k})$ containing no elements of the spectrum having $d_{k}-c_{k}\geq \frac{1}{N}|B_{k}-A_{k}|$.
\end{proof}
\end{theorem}
 
\begin{remark}
Note that the proof actually shows that the only numbers of the form $c^{i}\lim_{m\rightarrow \infty }\cc^{m}\phi _{0}^{(m-j)}\phi _{v}(z),$ for $z\in \lbrack 0,b]$, which may be contained in $(A_{k},B_{k})$, are those equal to $\lim_{m\rightarrow \infty }\cc^{m}\phi _{0}^{(m-i)}(z)$, where $k<i\leq k+n.$
\end{remark}

Next, we show that the three theorems in \cite{ZhGaps} can be deduced from Theorem~\ref{thm:crit}.

\begin{corollary}
\cite[Thm. 13]{ZhGaps} Suppose $b$ is the largest forbidden eigenvalue, and that $R^{-1}[0,\ b]\subseteq \lbrack 0,\ b],$ $\phi _{1}$ is decreasing on $[0,b]$, $\phi _{0}$ is strictly convex and $\phi _{0}(b)<\phi _{1}(b)$. Then there are gaps in the spectrum of the Laplacian.
\end{corollary}

\begin{proof}
We apply the theorem with $J=0$. Note that $\phi _{1}(b)\leq b$, and hence $\phi_{0}(b)<b$.
\end{proof}

\begin{corollary}
\cite[Thm. 16]{ZhGaps} Suppose $\alpha <\beta $ are two consecutive forbidden eigenvalues. Let $b$ be the largest forbidden eigenvalue and suppose that $c\geq b$ satisfies $R^{-1}([0,c])\subseteq \lbrack 0,c]$\footnote{In the statement of Theorem 16 of \cite{ZhGaps} it states $R^{-1}[0,b]\subseteq [0,c]$, but it is clear from the proof that the assumption $R^{-1}[0,c]\subseteq  [0,c]$ was intended.}. Assume $\phi_{0}$ is strictly convex, $\phi_{0}(c)\leq \alpha $ and $\phi _{1}(x)\geq \beta $ for all $x\in [0, c].$ Then there are gaps in the spectrum
of the Laplacian.
\end{corollary}

\begin{proof}
This also follows easily from Theorem \ref{thm:crit} since
\begin{equation*}
	\max \phi _{0}\mid _{[0,c]}\leq \alpha <\beta \leq \min \phi _{1}\mid_{[0,c]}\leq c.
\end{equation*}
\end{proof}

\begin{corollary}
\cite[Thm. 15]{ZhGaps} Suppose $a<b$ are the two largest forbidden
eigenvalues, $R^{-1}[0,\ b]\subseteq [0, a],$ $\phi _{1}$ is
decreasing on $[0,b]$ and $\phi _{0}$ is strictly convex. Then there are
gaps in the spectrum of the Laplacian.
\end{corollary}

\begin{proof}
Let $E(M,M_0)^{\prime }=\{\beta _{t}\}$ consist of all the elements in $E(M,M_0)\setminus \{b\}$ together with the real numbers $\phi _{0}(b), \ \cdots ,  \phi _{L}(b).$ As $R^{-1}[0, b]\subseteq [0, a],$ $\phi _{j}(b)\leq a<b$ for all $j=0,\cdots ,L.$ Hence the largest member of $E(M,M_0)^{\prime }$ is $a$ and $R^{-1}[0, a]\subseteq R^{-1}[0,
b]\subseteq [0, a]$.

All eigenvalues of the Laplacian are of the form $c^{i}\lim_{m\rightarrow \infty }\cc^{m}\phi _{0}^{(m-j)}\phi _{v}(\beta_{t}),$ thus we may apply the same arguments as in Theorem \ref{thm:crit}, but with $E(M,M_0)^{\prime }$ taking the place of the forbidden eigenvalues $E(M,M_0)$.

By the monotonicity assumptions, $\phi _{0}(a)=\max\limits_{[0,a]} \phi _{0} $ and $\phi _{1}(a)=\min\limits_{[0,a]}\phi _{1}$. Moreover, $\phi _{0}(a)<\phi _{0}(b)\leq \phi _{1}(b)<\phi _{1}(a)\leq a,$ thus Theorem \ref{thm:crit} implies there are a bounded number of elements in $(A_{k}, B_{k}),$ where 
\begin{eqnarray*}
	A_{k} &=&\cc^{k}\lim_{m\rightarrow \infty }\cc^{m}\phi _{0}^{(m-1)}\phi
_{0}(a)\text{ and} \\
	B_{k} &=&\cc^{k}\lim_{m\rightarrow \infty }\cc^{m}\phi _{0}^{(m-1)}\phi
_{1}(a).
\end{eqnarray*}
\end{proof}


\begin{thebibliography}{99}

\bibitem{ADT} E. Akkermans,  G. Dunne, A. Teplyaev
\emph{Physical Consequences of Complex Dimensions of Fractals,}\   
Europhys. Lett. {\bf88}, 40007 (2009).

\bibitem{ASST} B. Adams, S. A. Smith, R. Strichartz and A. Teplyaev
The spectrum of the {L}aplacian on the pentagasket,\
\emph{Fractals in Graz 2001}, 1–24, Trends Math., Birkhäuser, Basel, 2003. 

\bibitem
{eigenpapers} 
N Bajorin, T Chen, A Dagan, C Emmons, M Hussein, M~Khalil, P Mody, B Steinhurst, A~Teplyaev,
\emph{Vibration modes of $3n$-gaskets and other fractals,}  \ J. Phys. A: Math Theor. \textbf{41} (2008) 015101 (21pp); 
\emph{Vibration Spectra of Finitely Ramified, Symmetric Fractals},
 Fractals {\bf16} (2008), 243--258.
 
 \bibitem{BBK} M. T. Barlow, R. F. Bass and T. Kumagai, \emph{Stability of
parabolic Harnack inequalities on metric measure spaces.}\ J. Math.
Soc. Japan {\bf58} (2006) 485-519.

\bibitem{BBKT} M. T. Barlow, R. F. Bass, T. Kumagai, and A. Teplyaev, \emph{Uniqueness of Brownian motion on Sierpinski carpets.}\ 
J. Eur. Math. Soc. 
{\bf12} (2010), 655-701. 

\bibitem{BE} M.T. Barlow and S.N. Evans, 
\emph{Markov processes on vermiculated spaces.}\ Random walks and geometry, 337-348,  de Gruyter, Berlin, 2004. 

\bibitem{BK} M. T. Barlow and J. Kigami,
\emph{Localized eigenfunctions of the
Laplacian on p.c.f. self-similar sets.}\
J. London Math. Soc.,  \textbf{56} (1997),  320--332.

\bibitem{Br} H. Brolin, \emph{Invariant sets under iteration of rational
functions,}\,
 Ark. Mat. \textbf{6} (1965), 103--144. 
 
\bibitem{CG} L. Carleson and T.W. Gamelin,
\emph{Complex dynamics.}\ Universitext: Tracts in Mathematics. Springer-Verlag, New York, 1993.


\bibitem{DGV} G. Derfel, P. Grabner and F. Vogl, \emph{The zeta function of
the Laplacian on certain fractals}.  Trans. Amer. Math. Soc.  {\bf360}  (2008),   881--897.

\bibitem{DGV2} G. Derfel, P. Grabner and F. Vogl, \emph{Complex asymptotics of Poincar\'e functions and properties of Julia sets}.    Math. Proc. Cambridge Philos. Soc.  {\bf145}  (2008),   699--718.

\bibitem{DS} S. Drenning and R. Strichartz,
\emph{Spectral Decimation on Hambly's Homogeneous
Hierarchical Gaskets.} 
Illinois J. Math. {\bf53} (2009),  915-937 (2010).

\bibitem{FordSteinhurst} D. Ford and B. Steinhurst, \emph{Vibration Spectra of the $m$-Tree Fractal}, Fractals \textbf{18} (2010), 157-169.


\bibitem{FS} Fukushima, M. and Shima, T., \textit{On a spectral analysis for
the Sierpinski gasket}, Potential Anal. \textbf{1}(1992), 1-35.

\bibitem{GRS} Gibbons, M., Raj, A. and Strichartz, R., \textit{The finite
element method on the Sierpinski gasket}, Constructive Approx. \textbf{17} (2001), 561-588.

\bibitem{GibsonMacKenzie} S. Gibson and M. MacKenzie, \emph{Spectral analysis on a regular non-p.c.f. analog of the Sierpinski gasket}, preprint.



\bibitem{Grabner} P. Grabner, \emph{Functional iterations and stopping times
for Brownian motion on the Sierpi\'nski gasket.}\ Mathematika
\textbf{44} (1997), 374--400.

\bibitem{GW97} P. Grabner and W. Woess 
\emph{Functional iterations and periodic oscillations for simple random walk on the Sierpi\'nski graph.} 
Stochastic Process. Appl. {\bf69} (1997), no. 1, 127--138. 

\bibitem{GNS} R. Grigorchuk, V. Nekrashevych, and Z. Sunic,
\emph{Analysis on the First Julia Fractal}.
4th Conference on Analysis, Probability and Mathematical Physics on Fractals, Cornell University, Ithaca NY, 2011.

 \bibitem{GT} A. Grigor'yan and A. Telcs, 
\emph{Two-sided estimates of heat kernels on metric measure spaces}, to appear in Annals of Probability, 2011, 70 pages.

 \bibitem{HK} B. M. Hambly and T. Kumagai,  \emph{Transition density estimates for diffusion processes on post critically finite self-similar fractals.}\ Proc. London Math. Soc. (3)  \textbf{78} (1999), 431-458. 

\bibitem{HK-2010} B. M. Hambly and T. Kumagai,  \emph{Diffusion on the scaling limit of the critical percolation cluster in the diamond hierarchical lattice.}\ Comm. Math. Phys. \textbf{295} (2010),  29-69.

\bibitem{Ki} Kigami, J., \textit{Harmonic calculus on \pcf\ self-similar sets}, Trans. Amer. Math. Soc. \textbf{335} (1993), 721-755.

\bibitem{Ki2} Kigami, J., \textit{Analysis on fractals}, Cambridge U. Press,
New York, 2001.

\bibitem{Ki3} Kigami, J., \textit{Harmonic analysis for resistance forms}, J. Func. Anal. \textbf{204} (2003) 399-444.


 \bibitem{Ki09} J. Kigami,
\emph{Volume doubling measures and heat kernel estimates on self-similar sets.}
Mem. Amer. Math. Soc. {\bf199} (2009), no. 932.


\bibitem{KL1} J. Kigami and M. L. Lapidus,
\emph{Weyl's problem for the spectral distribution of Laplacians on p.c.f.
self-similar fractals.}\
Comm. Math. Phys. {\bf 158} (1993), 93--125.

\bibitem{Kr2002} B.~Kr\"{o}n,     \emph{Green functions on
self-similar graphs and bounds for the spectrum of
 the Laplacian.}
 Ann. Inst. Fourier (Grenoble)  {\bf 52}  (2002),   1875--1900.


\bibitem{KrTe03} B.~Kr\"{o}n and E.~Teufl, {\it Asymptotics of the
transition probabilities of the simple random walk on self-similar
graph,} Trans.\ Amer.\ Math.\ Soc., {\bf356} (2003) 393--414.

\bibitem{Kumagai-1993} T.~Kumagai, {\it Regularity, closedness and spectral dimensions of the Dirichlet forms on P.C.F. self-similar sets.} 
J. Math. Kyoto Univ. {\bf33} (1993),  765-786.


\bibitem{LvF2006} M. L. Lapidus and M. van Frankenhuysen, \emph{Fractal
geometry, complex dimensions and zeta functions. Geometry and
spectra of fractal strings.} Springer Monographs in Mathematics.
Springer, New York, 2006.



\bibitem{MT} Malozemov, L. and Teplyaev, A., \textit{Self-similarity,
operators and dynamics}, Math. Phys. Anal. Geom.\textbf{\ 6} (2003), 201-218.

\bibitem{Milnor} J. Milnor,
\emph{Dynamics in one complex variable.}\ 
Third edition. Annals of Mathematics Studies, {\bf160}. Princeton University Press,  2006.

\bibitem{NT} V. Nekrashevych and A. Teplyaev, \emph{Groups and analysis on fractals}, \emph{Analysis on graphs and its applications}, 143--180, Proc. Sympos. Pure Math., vol. 77, Amer. Math. Soc., Providence, RI, 2008.


\bibitem{Sabot2003} C. Sabot. \emph{Spectral properties of self-similar lattices and iteration of rational maps,} M\'em. Soc. Math. Fr. (N.S.), \textbf{92} (2003).

\bibitem{Sh} T. Shima, \textit{On eigenvalue problems for Laplacians on pcf
self-similar sets,} Japan J. Ind. Appl. Math. \textbf{13} (1996), 1-23. 

\bibitem{Steinhurst-BE} Benjamin Steinhurst.
\newblock {\em Diffusions and {L}aplacians on {L}aakso, {B}arlow-{E}vans, and
  other fractals}.
\newblock ProQuest LLC, Ann Arbor, MI, 2010.
\newblock Thesis (Ph.D.)--University of Connecticut.


\bibitem{ST-BE} B. Steinhurst and A. Teplyaev,
\emph{Symmetric Dirichlet forms and spectral analysis on Barlow-Evans fractals,}\
preprint, (2011).

\bibitem{ST-merozeta} B. Steinhurst and A. Teplyaev,
\emph{Existence of a meromorphic extension of spectral zeta functions on fractals,}\
preprint, (2011) arXiv:1011.5485.

\bibitem{St03jfa} R. S. Strichartz, \emph{Function spaces on fractals.} J.
Funct. Anal. {\bf198} (2003), 43--83.

\bibitem{St} Strichartz, R., \textit{Laplacians on fractals with spectral
gaps have nicer Fourier series}, Math. Res. Lett. \textbf{12} (2005), 269-274.

\bibitem{StBook} Strichartz, R., \textit{Differential equations on fractals:
a tutorial}, Princeton Univ. Press, New Jersey, 2006.


\bibitem{Te} Teplyaev, A., \textit{Spectral analysis on infinite Sierpinski
gasket}, J. Funct. Anal. \textbf{159} (1998), 537-567.

\bibitem{T07tams} \mbox{A. Teplyaev,} \emph{Spectral zeta functions of
fractals and the complex dynamics of polynomials.} Trans. Amer. Math. Soc.  {\bf 359}  (2007), 4339--4358.

\bibitem{T08cjm} \mbox{A. Teplyaev,} \emph{Harmonic coordinates on fractals
with finitely ramified cell structure.}   
Canad. J. Math.  {\bf60}  (2008),   457--480.

\bibitem{ZhGaps} Zhou, D., \textit{Criteria for spectral gaps of Laplacians
on fractals}.
 J. Fourier Anal. Appl.  \textbf{16}  (2010),   76--96. 

\bibitem{ZhVS} Zhou, D., \textit{Spectral analysis of Laplacians on the
Vicsek sets}.
     Pacific J. Math.  \textbf{241}  (2009),   369--398. 
     
\end{thebibliography}
\end{document}